\documentclass[12pt]{article}
\usepackage{amsmath,amsthm,amssymb}
\usepackage{mathrsfs}
\DeclareMathOperator{\qdim}{qdim}
\DeclareMathOperator{\Irr}{Irr}
\DeclareMathOperator{\Gal}{Gal}
\DeclareMathOperator{\ch}{ch}
\DeclareMathOperator{\tr}{tr}
\DeclareMathOperator{\diag}{diag}
\DeclareMathOperator{\glob}{glob}
\begin{document}
\input amssym.def
\setcounter{equation}{0}
\newcommand{\wt}{\mbox{wt}}
\newcommand{\spa}{\mbox{span}}
\newcommand{\Res}{\mbox{Res}}
\newcommand{\End}{\mbox{End}}
\newcommand{\Ind}{\mbox{Ind}}
\newcommand{\Hom}{\mbox{Hom}}
\newcommand{\Mod}{\mbox{Mod}}
\newcommand{\m}{\mbox{mod}\ }
\renewcommand{\theequation}{\thesection.\arabic{equation}}
\numberwithin{equation}{section}

\def \End{{\rm End}}
\def \Aut{{\rm Aut}}
\def \Z{\mathbb Z}
\def \M{\Bbb M}
\def \C{\mathbb C}
\def \R{\mathbb R}
\def \Q{\mathbb Q}
\def \N{\mathbb N}
\def \ann{{\rm Ann}}
\def \<{\langle}
\def \o{\omega}
\def \O{\Omega}
\def \M{{\cal M}}
\def \1t{\frac{1}{T}}
\def \>{\rangle}
\def \t{\tau }
\def \a{\alpha }
\def \e{\epsilon }
\def \l{\lambda }
\def \L{\Lambda }
\def \g{\gamma}
\def \b{\beta }
\def \om{\omega }
\def \o{\omega }
\def \cg{\chi_g}
\def \ag{\alpha_g}
\def \ah{\alpha_h}
\def \ph{\psi_h}
\def \nor{\vartriangleleft}
\def \V{V^{\natural}}
\def \voa{vertex operator algebra\ }
\def \voas{vertex operator algebras}
\def \v{vertex operator algebra\ }
\def \1{{\bf 1}}
\def \be{\begin{equation}\label}
\def \ee{\end{equation}}
\def \qed{\mbox{ $\square$}}
\def \pf {\noindent {\bf Proof:} \,}
\def \bl{\begin{lem}\label}
\def \el{\end{lem}}
\def \ba{\begin{array}}
\def \ea{\end{array}}
\def \bt{\begin{thm}\label}
\def \et{\end{thm}}
\def \br{\begin{rem}\label}
\def \er{\end{rem}}
\def \ed{\end{de}}
\def \bp{\begin{prop}\label}
\def \ep{\end{prop}}

\newtheorem{th1}{Theorem}
\newtheorem{ree}[th1]{Remark}
\newtheorem{thm}{Theorem}[section]
\newtheorem{prop}[thm]{Proposition}
\newtheorem{coro}[thm]{Corollary}
\newtheorem{lem}[thm]{Lemma}
\newtheorem{rem}[thm]{Remark}
\newtheorem{de}[thm]{Definition}
\newtheorem{hy}[thm]{Hypothesis}
\newtheorem{conj}[thm]{Conjecture}
\newtheorem{ex}[thm]{Example}

\begin{center}
{\Large {\bf Quantum Dimensions and Quantum Galois Theory}}\\
\vspace{0.5cm}

Chongying Dong\footnote
{Supported by NSF grants and a faculty research fund from the University of California at Santa Cruz.} \ \ Xiangyu Jiao
\\
Department of Mathematics, University of
California, Santa Cruz, CA 95064\\

Feng Xu\footnote{Supported by a NSF grant and a faculty research fund from the University of California at Riverside.}\\
Department of Mathematics, University of California, Riverside, CA 92521
\hspace{1.5 cm}
\end{center}

\begin{abstract} The quantum dimensions of modules for vertex operator algebras are defined and their properties are discussed. The possible values of the quantum dimensions are obtained for rational vertex operator algebras. A criterion for simple currents of a rational \voa is given. A full Galois theory for
rational vertex operator algebras is established using the quantum dimensions.
\end{abstract}

\section{Introduction}

The dimension of a space or an object is definitely an important concept in both mathematics and physics.
The goal of the present paper is to give a systematic study of the ``dimension'' of a module over a vertex operator algebra. More precisely, we study how to define quantum dimensions, how to compute quantum dimensions and the possible values of quantum dimensions. The concept of quantum dimensions goes back to the physical literature \cite{FMS} and the most discussions on quantum dimensions focus on the rational vertex operator algebras associated to the affine Kac-Moody algebras \cite{FZ}, \cite{DL} and Virasoro algebras \cite{DMZ}, \cite{W}. The mathematical work on quantum dimensions has been limited to the conformal nets approach to conformal field theory \cite{KLM} where the quantum dimensions are called the statistical dimensions or the square root of index \cite{J}, \cite{PP}.

Our own motivation for studying quantum dimensions comes from trying to understand the Galois theory
for vertex operator algebras \cite{DM1}, \cite{DLM}, \cite{HMT}, originated from orbifold theory
\cite{FLM}, \cite{DVVV}, \cite{DLM4}. For a vertex operator algebra $V$ and a finite automorphism group $G$ of $V,$
the fixed points $V^G$ is also a vertex operator algebra. It has already been established in \cite{DM1} and \cite{HMT} that there is a one to one correspondence between the subgroups of $G$ and vertex operator subalgebras of $V$ containing $V^G.$ To get a complete Galois theory for $V^G\subset V,$ one needs a notion of ``dimension'' $[V:V^G]$ of $V$ over $V^G$ such that $[V:V^G]=o(G).$ Various efforts were tried
without success until we turned our attention to the quantum dimensions. So as an application of the quantum dimensions we exhibit Galois theory for a vertex operator algebra $V$ together with a finite automorphism group $G.$

Let $V=\oplus_{n\in \Z}V_n$ be a vertex operator algebra and $M=\oplus_{n\geq 0}M_{\lambda+n}$ a $V$-module.  The quantum dimension $\qdim_VM$ of $M$ over $V$ is essentially the relative dimension $\frac{\dim M}{\dim V}.$ Unfortunately, both $\dim M$ and $\dim V$ are infinite. One has to use limits to approach  $\frac{\dim M}{\dim V}.$ The original definition of $\qdim_VM$ involves the $q$-characters of $V$ and $M.$ The $q$-character of $M$ is a formal power series
$$\ch_qM =  q^{\lambda-c/24}\sum_{n=0}^{\infty}(\dim M_{\lambda+n})q^{n},$$
where $c$ is the central charge of $V$. It is proved in \cite{Z} that the formal power series
$\ch_qM$ converges to a holomorphic function in the upper half plane in variable $\tau$ where $q=e^{2\pi i\tau}$ under certain conditions.  It is well known that $\qdim_VM$ can be  defined as the limit
of $\frac{\ch_qM}{\ch_qV}$ as $q$ goes to $1$ from the left. The advantage of this definition is that
one can use the modular transformation property of the $q$-characters \cite{Z} and Verlinde formula
\cite{V}, \cite{H} to compute the quantum dimensions and investigate their properties. In fact, we are following this approach closely in the present paper. The disadvantage of this definition is that it requires
both rationality and $C_2$-cofiniteness of $V.$  So this definition does not apply to irrational vertex operator algebras.

We propose two more definitions of quantum dimension, which work for any vertex operator algebra. The first one is given by the the limit of relative dimension $\lim_{n\to\infty}\frac{\dim M_{\lambda+n}}{\dim V_n}$ and the other is $\lim_{n\to\infty}\frac{\sum_{m=0}^n\dim M_{\lambda+m}}{\sum_{m=0}^n\dim V_m}.$ It is proved that if $\lim_{n\to\infty}\frac{\dim M_{\lambda+n}}{\dim V_n}$ exits then all the limits exist and are equal. The later definition of quantum dimension tells us the real meaning of the quantum dimension. We firmly believe that
these three definitions are equivalent although we could not prove the assertion in this paper. We also suspect that for a simple vertex operator algebra, the quantum dimension of any irreducible module exists.

The quantum dimensions for rational and $C_2$-cofinite vertex operator algebras have two main properties: (a) The quantum dimension of tensor product of two modules is the product of quantum dimensions; (b) An irreducible module is a simple current if and only if the quantum dimension is 1. The first property which is true for the tensor product of vector spaces is important in many aspects. This implies that the quantum dimensions satisfy a certain system of equations of degree 2 with integral coefficients and helps to compute the quantum dimensions. The second property enables us to determine the simple currents when the quantum dimensions are easily calculated. For example, for a framed vertex operator algebra \cite{DGH}, this can be easily done using the quantum dimensions for the rational vertex operator algebra $L(\frac12,0)$ associated to the Virasoro algebra with the central charge $\frac12.$

It is worthy to mention that the minimal weight $\lambda_{\min}$ of the irreducible modules plays an essential role in computing the quantum dimensions  using the $S$-matrix. For most rational vertex operator algebras including those associated to the unitary minimal series for the Virasoro algebra, $\lambda_{\min}=0$ is the weight of vertex operator algebra $V.$ The formula of the quantum dimensions in terms of $S$ matrix is more or less standard. For an arbitrary vertex operator algebra such as those associated to the non-unitary minimal series for the Virasoro algebra, $\lambda_{\min}$ can be negative. We obtain a similar formula for the quantum dimensions using the irreducible module
whose weight is $\lambda_{\min}$ instead of $V.$ We also give two examples of non-unitary vertex operator algebras to illustrate this.

Finding the possible values of the quantum dimensions for a rational and $C_2$-cofinite vertex operator algebra is another task in this paper. With the help of  Perron-Frobenius Theorem on eigenvalues and some graph theory we show that the quantum dimension of an irreducible module can only take values in $\{2\cos (\pi/n)|n\geq 3\}\cup [2,\infty)$ which are the square roots of the possible values of the index of subfactors of type II \cite{J}. The connection between quantum dimensions and index of subfactors is exciting but not surprising. There have been three approaches (algebraic, analytic and geometric) to two dimensional conformal field theory \cite{BPZ}, \cite {MS} in mathematics. The basic tool in algebraic approach is the vertex operator algebra and the analytic approach uses the conformal nets \cite{GL}, \cite{Wa}. Although the connection among different approaches has not been understood fully, constructing  a conformal net from a rational vertex operator algebra and a rational vertex operator algebra from a conformal net is highly desirable. The connection between quantum dimension and index gives further evidence for the equivalence of algebraic and analytic approaches to two dimensional conformal field theory.

Motivated by the representation theories of both finite groups and finite dimensional associative algebras, the notion of global dimension for a vertex operator algebra is proposed using the quantum dimensions of irreducible modules. Although we have not done much in the present paper on global dimension, the application of global dimension in classification of irreducible modules for orbifold and coset vertex operator algebras is visible. The main challenge is how to find an alternating definition without using the quantum dimensions of irreducible modules. One could classify the irreducible modules using the global dimension and the quantum dimensions of known irreducible modules. This will be very useful in studying the orbifold theory as in the case of conformal nets \cite{Xu}.

The paper is organized as follows. We give the basics including the definition of twisted modules and important concepts such as rationality, regularity and $C_2$-cofiniteness \cite{Z}, \cite{DLM1} in Section 2. The quantum dimensions are defined in Section 3 with examples. Section 4 is devoted to the study of the properties of quantum dimensions. In particular, the connection between quantum dimension and modular invariance \cite{Z}, tensor product of modules \cite{H} and Verlinde formula \cite{V} is investigated in great depth. A characterization of simple currents in terms of quantum dimensions is given. We present results on possible values of the quantum dimensions in Section 5. In the last section we give a full Galois theory for a simple vertex operator algebra with a finite automorphism group.

We thank Terry Gannon for useful suggestions on the possible values of the quantum dimensions.

\section{Preliminary}

In this section, we recall the various notions of twisted modules for a \voa following \cite{DLM1}. We also define the terms rationality, regularity, and $C_2$-cofiniteness from \cite{Z} and \cite{DLM1}. Besides, we discuss the modular invariance property of the trace functions for a rational \voa\cite{Z}.

\subsection{Basics}

A \voa $V=\oplus_{n\in \Z}V_n$ (as defined in \cite{FLM}) is said to be of CFT type if $V_n=0$ for negative
$n$ and $V_0=\C {\bf 1}.$

Let $V$ be a vertex operator algebra and $g$ an automorphism of $V$ with finite order $T$. Decompose $V$ into eigenspaces of $g:$
\begin{equation*}\label{g2.1}
V=\bigoplus_{r\in \Z/T\Z}V^r,
\end{equation*}
where $V^r=\{v\in V|gv=e^{-2\pi ir/T}v\}$.
We use $r$ to denote both
an integer between $0$ and $T-1$ and its residue class \m $T$ in this
situation. Let $W\{z\}$ denote the
space of $W$-valued formal series in arbitrary complex
 powers of $z$ for a vector
space $W.$

\begin{de} \label{weak}
A {\em weak $g$-twisted $V$-module} $M$ is a vector space equipped
with a linear map
\begin{equation*}
\begin{split}
Y_M: V&\to (\End\,M)\{z\}\\
v&\mapsto\displaystyle{ Y_M(v,z)=\sum_{n\in\Q}v_nz^{-n-1}\ \ \ (v_n\in
\End\,M)},
\end{split}
\end{equation*}
which satisfies the following:  for all $0\leq r\leq T-1,$ $u\in V^r$, $v\in V,$
$w\in M$,
\begin{eqnarray*}
& &Y_M(u,z)=\sum_{n\in \frac{r}{T}+\Z}u_nz^{-n-1} \label{1/2},\\
& &u_lw=0~~~ 				
\mbox{for}~~~ l\gg 0,\label{vlw0}\\
& &Y_M({\mathbf 1},z)=Id_M,\label{vacuum}
\end{eqnarray*}
 \begin{equation*}\label{jacobi}
\begin{array}{c}
\displaystyle{z^{-1}_0\delta\left(\frac{z_1-z_2}{z_0}\right)
Y_M(u,z_1)Y_M(v,z_2)-z^{-1}_0\delta\left(\frac{z_2-z_1}{-z_0}\right)
Y_M(v,z_2)Y_M(u,z_1)}\\
\displaystyle{=z_2^{-1}\left(\frac{z_1-z_0}{z_2}\right)^{-r/T}
\delta\left(\frac{z_1-z_0}{z_2}\right)
Y_M(Y(u,z_0)v,z_2)},
\end{array}
\end{equation*}
where $\delta(z)=\sum_{n\in\Z}z^n$ (elementary
properties of the $\delta$-function can be found in \cite{FLM}) and
all binomial expressions (here and below) are to be expanded in nonnegative
integral powers of the second variable.
\end{de}

\begin{de}\label{ordinary}
A $g$-{\em twisted $V$-module} is
a weak $g$-twisted $V$-module $M$ which carries a
$\C$-grading induced by the spectrum of $L(0)$ where $L(0)$ is the component operator of $Y(\omega,z)=\sum_{n\in \Z}L(n)z^{-n-2}.$ That is, we have
\begin{equation*}
M=\bigoplus_{\lambda \in{\C}}M_{\lambda},
\end{equation*}
where $M_{\l}=\{w\in M|L(0)w=\l w\}.$ Moreover we require that
$\dim M_{\l}$ is finite and for fixed $\l,$ $M_{\frac{n}{T}+\l}=0$
for all small enough integers $n.$
\end{de}

In this situation, if $w\in M_{\l}$ we refer to $\l$ as the {\em weight} of
$w$ and write $\l=\wt w.$  The totality of $g$-twisted $V$-modules defines a full subcategory of the category of $g$-twisted weak
$V$-modules.

Denote the set of nonnegative integers by $\Z_+.$
An important and related class of g-twisted modules is the following.

\begin{de}\label{admissible}
 An {\em admissible} $g$-twisted $V$-module
is a  weak $g$-twisted $V$-module $M$ that carries a
$\frac1T{\Z}_{+}$-grading
\begin{equation*}
M=\bigoplus_{n\in\frac{1}{T}\Z_+}M(n),
\end{equation*}
which satisfies the following
\begin{equation*}
v_mM(n)\subseteq M(n+\wt v-m-1)
\end{equation*}
for homogeneous $v\in V,$ $m\in \frac{1}{T}{\Z}.$
\ed

If $g=Id_V$  we have the notions of  weak, ordinary and admissible $V$-modules \cite{DLM1}.

If $M=\bigoplus_{n\in \frac{1}{T}\Z_+}M(n)$
is an admissible $g$-twisted $V$-module, the contragredient module $M'$
is defined as follows:
\begin{equation*}
M'=\bigoplus_{n\in \frac{1}{T}\Z_+}M(n)^{*},
\end{equation*}
where $M(n)^*=\Hom_{\C}(M(n),\C).$ The vertex operator
$Y_{M'}(a,z)$ is defined for $a\in V$ via
\begin{eqnarray*}
\langle Y_{M'}(a,z)f,u\rangle= \langle f,Y_M(e^{zL(1)}(-z^{-2})^{L(0)}a,z^{-1})u\rangle,
\end{eqnarray*}
where $\langle f,w\rangle=f(w)$ is the natural paring $M'\times M\to \C.$
One can prove (cf. \cite{FHL}, \cite{X}) the following:
\bl{l3.8} $(M',Y_{M'})$ is an admissible $g^{-1}$-twisted $V$-module.
\el
We can also define the contragredient module $M'$ for a $g$-twisted $V$-module $M.$ In this case,
$M'$ is a $g^{-1}$-twisted $V$-module. Moreover, $M$ is irreducible if and only if $M'$ is irreducible.

\begin{de}
A \voa $V$ is called $g$-rational, if the  admissible $g$-twisted module category is semisimple. $V$ is called rational if $V$ is $1$-rational.
\end{de}

The following lemma about $g$-rational \voas \  is well-known \cite{DLM2}.
\begin{lem}
If $V$ is $g$-rational and $M$ is an irreducible admissible $g$-twisted $V$-module, then

(1) $M$ is an $g$-twisted $V$-module and  there exists a number $\l \in \mathbb{C}$ such that  $M=\oplus_{n\in \frac{1}{T}\mathbb{Z_+}}M_{\l +n}$ where $M_{\lambda}\neq 0.$ The $\l$ is called the conformal weight of $M;$

(2) There are only finitely many irreducible admissible  $g$-twisted $V$-modules up to isomorphism.
\end{lem}

For a simple \voa $V$  which has finitely many irreducible modules, we always denote $M^0, M^1, \cdots, \ M^d$ all the inequivalent irreducible $V$-modules with $M^0\cong V$. And let
 $\lambda_i$ denote the conformal weight of $M^i.$ The following theorem  is proved in \cite{DLM4}.
\begin{thm}
Let $V$ be a rational and $C_2$-cofinite vertex operator algebra, then $\l_i\in \Q$, $\forall i=0,\cdots, d.$
\end{thm}

Besides rationality, there is another important concept called $C_2$-cofiniteness \cite{Z}.
\begin{de}
We say that a VOA $V$ is $C_2$-cofinite if $V/C_2(V)$ is finite dimensional, where $C_2(V)=\langle v_{-2}u|v,u\in V\rangle.$
\end{de}

\begin{de}
A \voa $V$ is called regular if every weak $V$-module is a direct sum of irreducible $V$-modules.
\end{de}

\begin{rem}
It is proved in \cite{ABD} that if  $V$ is of CFT type, then regularity is equivalent to rationality and $C_2$-cofiniteness. Also $V$ is regular if and only if the weak module category is semisimple \cite{DY}.
\end{rem}

Let $V$ be a vertex operator algebra. In \cite{DLM3} a series of associative algebras $A_n(V)$ were introduced for nonnegative integers $n$. In the case $n=0,$ $A_0(V)=A(V)$ is the Zhu's algebra as defined in \cite{Z}. We briefly review $A_n(V)$ here. For homogeneous  $u,~v\in V$ we define
\begin{equation*}
\begin{split}
u\circ_n v&=\Res_zY(u,z)v\frac{(1+z)^{wt u +n}}{z^{2n+2}},\\
u*_n v     &= \sum_{m=0}^n (-1)^m
\left(
\begin{array}{ccc}
 m+n \\
 n
\end{array}
\right)
\Res_z Y(u,z)\frac{(1+z)^{wt u+n}}{z^{n+m+1}}.\\
\end{split}
\end{equation*}
Extend $\circ_n$ and $*_n$ linearly to obtain  bilinear products on $V.$ We let $O_n(V)$ be the linear span of all $u\circ_n v$ and $L(-1)u+L(0)u .$ We have (see \cite{DLM3}, \cite{Z}):
\begin{thm} Let $V$ be a \voa  and $M=\oplus_{n=0}^\infty M(n)$ be an admissible $V$-module. Set  $A_n(V)=V/O_n(V).$ Then

(1) For any $n\in \Z_+, $  $A_n(V)$  is an associative algebra with respect to $*_n;$

(2) For $0\leq m\leq n,$ $M(m)$ is an $A_n(V)$-module;

(3) If $V$ is rational, then for any $n\in \Z_+,$ $A_n(V)$ is a finite dimensional semisimple associative algebra;

(4) If $V$ is a simple rational \voa and assume $\{M^i| i=0, \cdots, d\}$ be the inequivalent irreducible $V$-modules,  then $A_n(V)\cong \oplus_{i=0}^d\oplus_{m\leq n}\End M^i(m).$

\end{thm}

\subsection{Modular Invariance of Trace Functions}

We now turn our discussion to the modular-invariance property in VOA theory. The most basic function is the formal character of a $g$-twisted $V$-module $M=\oplus_{n\in\frac{1}{T}\Z_+} M_{\l+n}.$ We define the formal character of $M$ as
\begin{equation*}
\ch_q M=\tr_{M}q^{L(0)-c/24}=q^{\lambda-c/24}\sum_{n\in\frac{1}{T}\Z_+} (\dim M_{\l+n})q^n,
\end{equation*}
where $\lambda$ is the conformal weight of $M$. It is proved in \cite{Z} and \cite{DLM4} that
$\ch_qM$ converges to a holomorphic function on the domain $|q|<1$ if $V$ is $C_2$-cofinite.
We sometimes also use $Z_M(\tau)$ to denote the holomorphic function $\ch_q M.$
Here and below $\tau$ is in the complex upper half-plane $\textbf{H}$ and $q=e^{2\pi i\tau}.$

For any homogeneous element $v\in V$ we define a trace function associated to $v$ as follows:
\begin{equation*}
Z_M(v,\tau)=\tr_{M}o(v)q^{L(0)-c/24}=q^{\lambda-c/24}\sum_{n\in\frac{1}{T}\Z_+} \tr_{M_{\l+n}}o(v)q^n,
\end{equation*}
where $o(v)=v(\wt v-1)$ is the degree zero operator of $v$.

There is a natural action of $\Aut(V)$ on twisted modules \cite{DLM4}. Let $g, h$ be two automorphisms of $V$ with $g$ of finite order. If $(M, Y_g)$ is a weak $g$-twisted $V$-module, there is a weak $hgh^{-1}$-twisted  $V$-module $(M\circ h, Y_{hgh^{-1}})$ where $M\circ h\cong M$ as vector spaces and
\begin{equation*}
Y_{hgh^{-1}}(v,z)=Y_g(h^{-1}v,z)
\end{equation*}
for $v\in V.$
This defines a left action of $\Aut(V)$ on weak twisted $V$-modules and on isomorphism
classes of weak twisted $V$-modules. Symbolically, we write
\begin{equation*}
h\circ (M,Y_g)=(M\circ h,Y_{hgh^{-1}})=h\circ M,
\end{equation*}
where we sometimes abuse notation slightly by identifying $(M, Y_g)$ with the isomorphism
class that it defines.

If $g, h$ commute, obviously $h$ acts on the $g$-twisted modules as above.
We set $\mathscr{M}(g)$ to be the equivalence classes of irreducible $g$-twisted $V$-modules and $\mathscr{M}(g,h)=\{M \in \mathscr{M}(g)| h\circ M\cong M\}.$ Then for any $M\in \mathscr{M}(g,h),$ there is a $g$-twisted $V$-module isomorphism
\begin{equation*}
\varphi(h) : h\circ M\to M.
\end{equation*}
The linear map $\varphi(h)$ is unique up to a nonzero scalar.
We set
\begin{equation*}
Z_M(g,h,\tau)=\tr_{_M}\varphi(h) q^{L(0)-c/24}=q^{\lambda-c/24}\sum_{n\in\frac{1}{T}\Z_+}\tr_{_{M_{\l+n}}}\varphi(h)q^{n}.
\end{equation*}
Since $\varphi(h)$ is unique up to a nonzero scalar, $Z_M(g,h,\tau)$ is also defined up to a nonzero scalar. The choice of the scalar does not interfere with any of the results in this paper.
For a homogeneous element $v\in V$ and any commuting pair $(g,h)$ we define
\begin{equation*}
T_M(v,g,h,\tau)=q^{\l-c/24}\sum_{n\in\frac{1}{T}\Z_+}\tr_{_{M_{\l+n}}} o(v)\varphi(h) q^{n},
\end{equation*}
where $M\in \mathscr{M}(g,h)$. It is easy to see when $v=\1,$ $T_M(1,g,h,\tau)=Z_M(g,h,\tau).$

Zhu has introduced a second vertex operator algebra $(V, Y[~], \1, \tilde{\omega})$ associated to $V$
in \cite{Z}.  Here $\tilde{\omega}=\omega-c/24$ and
$$Y[v,z]=Y(v,e^z-1)e^{z\cdot \wt v}=\sum_{n\in \Z}v[n]z^{n-1}$$
for homogeneous $v.$ We also write
$$Y[\tilde{\omega},z]=\sum_{n\in \Z}L[n]z^{-n-2}.$$

We must take care to distinguish between the notion of conformal weight in the original \voa  and in the second \voa $(Y,Y[~],\1,\tilde{\omega}).$ If $v\in V$ is homogeneous in the second vertex operator algebra, we denote its weight by $\wt [v].$
For such $v$ we define an action of the modular group $\Gamma$ on $T_{M}$ in a familiar way, namely
\begin{equation*}
T_M|_\gamma(v,g,h,\tau)=(c\tau+d)^{-\wt[v]}T_M(v,g,h,\gamma \tau),
\end{equation*}
where $\gamma \tau$ is the M\"obius transformation, that is
\begin{equation}\label{e1.11}
\gamma: \tau\mapsto\frac{ a\tau + b}{c\tau+d},\ \ \ \gamma=\left(\begin{array}{cc}a & b\\ c & d\end{array}\right)\in\Gamma=SL(2,\Z).
\end{equation}
Let $P(G)$ denote the commuting pairs of elements in a group $G.$ We let $\gamma\in \Gamma$ act on the right of $P(G)$ via
$$(g, h)\gamma = (g^ah^c, g^bh^d ).$$
The following theorem is proved in \cite{Z}, \cite{DLM4}.
\begin{thm}\label{minvariance}
Assume $(g, h)\in P(\Aut(V))$ such that the orders of $g$ and $h$ are finite. Let $\gamma =\left(\begin{array}{cc}a & b\\ c & d\end{array}\right)\in \Gamma.$ Also assume that $V$ is $g^ah^c$-rational and $C_2$-cofinite. If $M^i$ is an irreducible $h$-stable $g$-twisted $V$-module, then
\begin{equation*}
T_{M^i}|_{\gamma}(v,g,h,\tau)=\sum_{N_j\in \mathscr{M}(g^ah^c,{g^bh^d)}} \gamma_{i,j}(g,h) T_{N_j}(v,(g, h)\gamma, ~\tau),
\end{equation*}
where $\gamma_{i,j}(g,h)$ are some complex numbers independent of the choice of $v\in V$.
\end{thm}

\begin{rem}\label{trans}
In the case  $g=h=1$, and $\gamma=S=\left(\begin{array}{cc}0 & -1\\ 1 & 0\end{array}\right)$ we
 have:
\begin{equation}\label{S-tran}
Z_{M^i}(v,-\frac{1}{\tau})=\tau^{\wt[v]}\sum_{j=0}^d S_{i,j} Z_{M^j}(v,\tau).
\end{equation}
The matrix $S=(S_{i,j})$ is called $S$-matrix which is independent of the choice of $v.$
\end{rem}

\section{Quantum Dimension}

In this section we define the  quantum dimension for an ordinary  $V$-module $M$ and give some examples.
\subsection{Definition}
\begin{de} Let $V$ be a \voa and $M$ a $V$-module such that $Z_V(\tau)$ and $Z_M(\tau)$ exist. The quantum dimension of $M$ over $V$ is defined as
\begin{equation}
\qdim_{V}M=\lim_{y\to 0}\frac{Z_M(iy)}{Z_V(iy)},
\end{equation}
where $y$ is real and positive.
\end{de}

\begin{rem} \label{rem def}
Sometimes we use an alternating definition which involves the $q$-characters:
\begin{equation}
\qdim_VM=\lim_{q\to 1^-}\frac{\ch_qM}{\ch_q V}.
\end{equation}
This is because we know as $\tau=iy \to 0,$ $q=e^{2\pi i \tau}=e^{-2\pi y} \to 1^-.$ This definition
of quantum dimension seems well known in the literature.
\end{rem}

\begin{rem}
Quantum dimension is formally defined. Intuitively, for an arbitrary \voa $V$ and any $V$-module $M$,  $\qdim_{V}M$ might not exist.   But we will prove that for rational and $C_2$-cofinite \voas, quantum dimensions do exist.
 \end{rem}

\begin{rem} \label{rem positive dim}
If $\qdim_{V}M$ exists, then it is nonnegative.
\end{rem}

\begin{rem}\label{M'}
If $\qdim_{V}M$ exists, then
\begin{equation}
\qdim_V M=\qdim_V M'.\\
\end{equation}
\end{rem}

In the definition of quantum dimensions one needs the convergence of both $\ch_qV$ and $\ch_qM.$ But
there is no theorem which guarantees the convergence of these formal $q$-characters for an arbitrary vertex operator algebra. It is natural to seek other equivalent definition of quantum dimensions without using the
convergence of formal characters.

\begin{prop}\label{equidef1}
Let $V=\oplus_{n=0}^{\infty}V_n$ be a vertex operator algebra, and let $M\!=\!\oplus_{n=0}^{\infty}M(n)$ be an admissible $V$-module with $\dim M(n)<\infty,$ $\forall n\geq 0.$  If $\lim_{n\to \infty} \frac{\dim M(n)}{\dim V_n}=d$ exists, then

(1) $\lim_{n\to \infty}\frac{\sum_{i=0}^n \dim M(i)}{\sum_{i=0}^n \dim V_i}=d,$

(2) $\lim_{q\to 1^-}\frac{\ch_qM}{\ch_qV}=d.$
\end{prop}

\pf For simplicity, we denote $\dim M(n)$ by $a_n$ and $\dim V_n$ by $b_n.$ We first deal with the case that $d\ne \infty.$  Then $\lim_{n\to \infty} \frac{a_n} {b_n}=d$ implies for any $\varepsilon>0,$ $\mid\frac{a_n}{b_n}-d\mid< \varepsilon,$ for $n\geq N,$ where N is a sufficiently   large integer. Thus we have
\begin{equation*}
\mid a_n- b_nd\mid<\varepsilon b_n.
\end{equation*}

For (1) we consider the following estimation:
\begin{equation*}\label{inequality1}
\begin{split}
\left|\frac{\sum_{n=0}^m a_n}{\sum_{n=0}^m b_n}-d\right|&= \left |\frac{\sum_{n=0}^m a_n-\sum_{n=0}^m b_nd}{\sum_{n=0}^m b_n}\right|\\
&= \left |\frac{\sum_{n=0}^m (a_n- b_nd)}{\sum_{n=0}^m b_n}\right|\\
&\leq \left |\frac{\sum_{n=0}^N (a_n- b_nd)}{\sum_{n=0}^m b_n}\right|+\left |\frac{\sum_{n=N+1}^m (a_n- b_nd)}{\sum_{n=0}^m b_n}\right|\\
&\leq \left |\frac{\sum_{n=0}^N (a_n- b_nd)}{\sum_{n=0}^m b_n}\right|+\left |\frac{\sum_{n=N+1}^m \varepsilon b_n}{\sum_{n=0}^m b_n}\right|\\
&\leq \left |\frac{\sum_{n=0}^N (a_n- b_nd)}{\sum_{n=0}^m b_n}\right|+\varepsilon.
\end{split}
\end{equation*}
For the first term above, we know  that $\lim_{m\to \infty}\sum_{n=0}^m b_n=\infty,$ and the top term $\sum_{n=0}^N (a_n- b_nd)$ is a fixed number. So we can take $m> N_1,$ for some big enough number $N_1,$ and we get
\begin{equation*}\label{inequality 2}
 \left |\frac{\sum_{n=0}^N (a_n- b_nd)}{\sum_{n=0}^m b_n}\right|<\varepsilon .
\end{equation*}
So for $m>\max\{N,N_1\},$  $\left|\frac{\sum_{n=0}^m a_n}{\sum_{n=0}^m b_n}-d\right|<2\varepsilon ,$ i.e. $$\lim_{n\to \infty}\frac{\sum_{i=0}^n \dim M(i)}{\sum_{i=0}^n \dim V_i}=d.$$

(2) is proved in a similar way:
\begin{equation*}
\begin{split}
\left|\frac{\sum_{n=0}^\infty a_n q^n}{\sum_{n=0}^\infty b_n q^n}-d\right|&=\left|\frac{\sum_{n=0}^\infty (a_n-b_nd) q^n}{\sum_{n=0}^\infty b_n q^n}\right|\\
&=\left|\frac{\sum_{n=0}^N (a_n-b_nd) q^n}{\sum_{n=0}^\infty b_n q^n}+\frac{\sum_{n=N+1}^\infty (a_n-b_nd) q^n}{\sum_{n=0}^\infty b_n q^n}\right|\\
&\leq \left|\frac{\sum_{n=0}^N (a_n-b_nd) q^n}{\sum_{n=0}^\infty b_n q^n}\right|+\sum_{n=N+1}^\infty\left|\frac{ (a_n-b_nd) q^n} {\sum_{n=0}^\infty b_n q^n}\right|\\
&\leq \left|\frac{\sum_{n=0}^N (a_n-b_nd) q^n}{\sum_{n=0}^\infty b_n q^n}\right|+\sum_{n=N+1}^\infty\left|\frac{ (\varepsilon b_n) q^n}{\sum_{n=0}^\infty b_n q^n}\right|\\
&\leq  \left|\frac{\sum_{n=0}^N (a_n-b_nd) q^n}{\sum_{n=0}^\infty b_n q^n}\right|+\varepsilon.
\end{split}
\end{equation*}
Obviously,  $\lim_{q\to 1^-}\left|\frac{\sum_{n=0}^N (a_n-b_nd) q^n}{\sum_{n=0}^\infty b_n q^n}\right| =0.$ Thus when $q$ is close to 1, the first term above is very small and so is
$\left|\frac{\sum_{n=0}^\infty a_n q^n}{\sum_{n=0}^\infty b_n q^n}-d\right|.$ This shows that
$\lim_{q\to 1^-}\frac{\ch_qM}{\ch_qV}=d.$

If $d=\infty,$ it is evident that $\lim_{n\to \infty}\frac{\sum_{i=0}^n \dim M(i)}{\sum_{i=0}^n \dim V_i}=\infty$ and  $\lim_{q\to 1^-}\frac{\ch_qM}{\ch_qV}=\infty.$ The proof is complete. \qed

\begin{rem}\label{remaa} We believe that for a simple vertex operator algebra $V$ and an irreducible $V$-module $M$,
if one of the three limits $\lim_{n\to \infty} \frac{\dim M(n)}{\dim V_n},$ $\lim_{n\to \infty}\frac{\sum_{i=0}^n \dim M(i)}{\sum_{i=0}^n \dim V_i}$ and  $\lim_{q\to 1^-}\frac{\ch_qM}{\ch_qV}$ exists and is finite, then the other limits also exist and all limits are equal. But we cannot establish this result in the paper.  There are counterexamples in pure analysis:
If $a_n,b_n$ are nonnegative real numbers for $n\geq 0$ such that $\lim_{n\to \infty}\frac{\sum_{m=0}^nb_m}{\sum_{m=0}^na_m}$ exists and is finite, but $\lim_{n\to \infty}\frac{b_n}{a_n}=\infty.$ So the equality of these three limits has some deep reason (which we do not know) from the theory of vertex operator algebra.
\end{rem}

If $V$ is a \voa with only finitely many inequivalent irreducible modules, say $ M^0, \cdots, M^d,$ we give the definition of global dimension.
\begin{de}
The global dimension of $V$ is defined as $\glob (V)=\sum_{i=0}^{d}(\qdim_V M^i)^2.$
\end{de}

\begin{rem}
In the index theory of conformal nets, the so called $\mu$-index which is similar to the global dimension defined above plays an important role in orbifold theory. We expect that we can use global dimensions to get some significant results in \voa theory. Same global dimension is also defined in the setting of fusion category \cite{ENO}.
\end{rem}

Here is a result on the global dimension for a rational vertex operator algebra using $A_n(V).$
\begin{prop} Let $V$ be a rational vertex operator algebra and $M^i=\oplus_{n\geq 0}M^i(n)$ be the irreducible modules with $M^i(0)\ne 0$ for all $i.$ We assume that $\lim_{n\to\infty}\frac{\dim M^i(n)}{\dim V_i}$ exists and is finite for all $i.$ Then
$$\glob (V)=\lim_{n\to\infty}\frac{\dim A_n(V)/A_{n-1}(V)}{(\dim V_n)^2}.$$
\end{prop}

\pf It follows from \cite{DLM3} that $A_n(V)\cong \oplus_{i=0}^d\oplus_{m\leq n}\End M^i(m).$ Then
$$A_n(V)/A_{n-1}(V)\cong \bigoplus_{i=0}^d\End M^i(n).$$
 Using Proposition \ref{equidef1} and the definition
of global dimension gives the result immediately. \qed

\subsection{Examples}

We use the definition to compute the quantum dimensions of modules  for Heisenberg vertex operator algebras and Virasoro vertex operator algebras.

\begin{ex}\label{hei} {\rm Let $M(1)$ be the Heisenberg \voa constructed from the vector space $H$ of dimension $d$ and with a nondegenerate symmetric bilinear form. It is well known that every irreducible $M(1)$-module is of the form $M(1,\l),$ $\l \in H$ with $q$-character
\begin{equation*}
\ch_q M(1,\l)= \frac {q^{\frac{(\l,\l)}{2}-\frac{d}{24}}}{\prod _{n\geq 1}(1-q^n)^d}.
\end{equation*}
Using the alternating definition in Remark \ref{rem def}   we have
\begin{equation*}
\qdim_{M(1)} M(1,\l)
=\lim_{q\to 1^-}   \frac{ \ \ \frac{q^{\frac{(\l,\l)}{2}}}  {\ \prod_{n\geq 1 } (1-q^n)^d \ }\ \  }  {\frac{1}{\  \prod_{n\geq 1} (1-q^n)^d\ }}\\
= 1.
\end{equation*}}
\end{ex}

 \begin{ex}\label{dim L(1,h)} {\rm
 Let $c$ and $h$ be two complex numbers and let $L(c,h)$ be the lowest weight irreducible module for the Virasoro algebra with central charge $c$ and lowest weight $h.$ Then $L(c,0)$ has a natural \voa structure \cite{FZ}. We are interested in the vertex operator algebra $L(1,0)$ and its irreducible modules. We know from \cite{KR} that
$$\dim_qL(1,h)=\left\{\begin{array}{ll}
\frac{q^{\frac{n^2}{4}}-q^{\frac{(n+2)^2}{4}}}{\eta(q)} & {\rm if} \
h=\frac{n^2}{4},n\in\Z,\\
\frac{q^h}{\eta(q)}& {\rm otherwise,}
\end{array}\right.$$
 where
$$\eta(q)=q^{1/24}\prod_{n=1}^{\infty}(1-q^n).$$
Then if $h\ne \frac{n^2}{4}$ for any integer $n,$
\begin{equation*}
\begin{split}
\qdim_{L(1,0)}L(1,h)=\lim_{q\to 1^-}\frac{~\frac{q^h}{~\prod _{n\geq1}(1-q^n)~}~}{~\frac{1}{~\prod _{n>1}(1-q^n)~}~}
=\lim_{q\to 1^-}\frac{q^h}{1-q}
=\infty.
\end{split}
 \end{equation*}
Similarly for $h=\frac{m^2}{4}$ for some integer $m,$
\begin{equation*}
\begin{split}
\qdim_{L(1,0)}L(1, \frac{m^2}{4})=&\lim_{q\to 1^-}\frac{q^{\frac{m^2}{4}}-q^{\frac{(m+2)^2}{4}}}{1-q}\\
=& \lim_{q\to 1^-}\frac{q^{\frac{m^2}{4}}(1-q^{m+1})}{1-q}\\
=& \lim_{q\to 1^-}q^{\frac{m^2}{4}}(1+q+\cdots+q^m)\\
=& m+1.
\end{split}
\end{equation*}}
\end{ex}

\section{Quantum Dimensions and the Verlinde Formula}

Computing quantum dimensions directly by using the definition is not easy. However, the definition of quantum dimensions involves the $q$-characters. This motivates us to use the modular invariant property of the trace functions \cite {Z} to do the computation. It turns out that the Verlinde formula \cite{V}, \cite{H} plays an important role.

\subsection{Verlinde Formula}

The Verlinde conjecture \cite{V} in conformal field theory states that the action of the modular transformation $\tau\to -1/\tau$ on the space of characters of a rational conformal field theory diagonalizes the fusion rules.  In this section, we quote some results  from \cite{H} about the Verlinde conjecture of  rational \voas.

Let $V$ be a rational and $C_2$-cofinite simple VOA, and let $M^0, M^1,\cdots,\ M^d$ be as before.
We use $N_{i,j}^k$ to denote $\dim I_V \left(
\begin{array}{c}
\ \ M^k \ \\
M^i \ \  M^j
\end{array}
\right),$ where
$I_V \left (
\begin{array}{c}
\ \ M^k \ \\
M^i \ \  M^j
\end{array}
 \right)$
 is the space of intertwining operators of type
 $\left(
\begin{array}{c}
\ \ M^k \ \\
M^i \ \  M^j
\end{array}
 \right).$ $N_{i, j}^k$ are called the fusion rules.
As usual, we use $M^{i'}$ to denote $(M^i)',$ the contragredient module of $M^i.$

 The following theorem which plays an essential role in this section is proved in \cite{H}.
\begin{thm}\label{Verlinde} Let $V$ be a rational and $C_2$-cofinite simple \voa of CFT type and assume $V\cong V'.$ Let $S=(S_{i,j})_{i,j=0}^d$ be the $S$-matrix as defined in (\ref{S-tran}). Then

(1) $(S^{-1})_{i,j}=S_{i,j'}=S_{i',j},$ and $S_{i',j'}=S_{i,j}; $

(2) $S$ is symmetric and $S^2=(\delta_{i,j'});$

(3) $N_{i,j}^k=\sum_{s=0}^d\frac{S_{j,s}S_{i,s}S^{-1}_{s,k}}{S_{0,s}};$

(4) The $S$-matrix diagonalizes the fusion matrix $N(i)=(N_{i,j}^k)_{j,k=0}^d$ with diagonal entries $\frac{S_{i,s}}{S_{0,s}},$ for  $i, s=0,\cdots,d.$ More explicitly, $S^{-1}N(i)S=\diag(\frac{S_{i,s}}{S_{0,s}})_{s=0}^d.$ In particular, $S_{0,s}\neq 0$ for $s=0,\cdots,d.$
\end{thm}

\subsection{Properties of Quantum Dimensions}

We prove that the quantum dimensions of modules exist and are related to the $S$-matrix under certain assumptions in this subsection. Using the Verlinde Formula we obtain an expression of the global dimension.
We also show that quantum dimensions are multiplicative under tensor product, and give a criterion for a simple module to be a simple current.

\begin{lem}\label{lem qdim}
Let $V$ be a simple, rational and $C_2$-cofinite \voa of CFT type. Let $M^0,\ M^1, \cdots ,\ M^d$ be as before with the corresponding conformal weights $\l_i>0$ for $0<i\leq d$. Then $0<\qdim_{V}M^i<\infty$ for any $0\leq i\leq d.$
\end{lem}
\begin{proof}
Since $V$ is rational and $C_2$-cofinite, we can use the modular transformation rule given in (\ref{S-tran}). By definition, we have
\begin{equation}\label{qdim S}
\begin{split}
\qdim_{V} M^i=&\lim_{y\to 0}\frac{Z_{M^i}(iy)}{Z_V(iy)}\\
=&\lim_{\tau \to i\infty}\frac{Z_{M^i}(-\frac{1}{\tau})}{Z_{V}(-\frac{1}{\tau})}\\
=&\lim_{\tau \to i\infty}\frac{\sum_{j=0}^{d} S_{i,j}Z_{M^j}(\tau)}{\sum_{j=0}^{d}S_{0,j}Z_{M^j}(\tau)}\\
=&\lim_{q\to 0^+}\frac{\sum_{j=0}^{d} S_{i,j}\ch_qM^j}{\sum_{j=0}^{d}S_{0,j}\ch_qM^j}\\
=&\frac{S_{i,0}}{S_{0,0}}.
\end{split}
\end{equation}
The last equality is true because the conformal weight $\lambda_i>0$ except $\lambda_0,$ which implies $\lim_{q\to 0^+} \ch_qM^j=0$ for $0<j\leq d.$
By Theorem \ref{Verlinde} we know that  $S_{i,0}\neq 0$ for $0\leq i\leq d.$
So $\qdim_VM^i$ exists for all $0\leq i\leq d.$ By Remark \ref{rem positive dim} we conclude $0<\qdim_{V}M^i<\infty$.
\end{proof}

\begin{rem}
The computation given in (\ref{qdim S}) indicates the quantum dimension of a $V$-module, in some sense, only depends on the $S$-matrix. From Remark \ref{trans} we know that the $S$-matrix does not depend on the element $v$ we choose. Therefore we can actually define $\qdim_V M$ in the following way:
$$\qdim_V M=\lim_{y\to 0}\frac{Z_M(v, iy)}{Z_V(v,iy)}$$
for any homogeneous $v\in V$ with $o(v)|_{V_0}\neq 0.$
\end{rem}

\begin{rem}\label{nonunitary} Lemma \ref{lem qdim} can be generalized as follows. Assume $V$ is a simple, rational and $C_2$-cofinite VOA with $V\cong V'.$ Let $M^i$ be as before.  Also assume $\l_k=\lambda_{\min}=\min_i\{\l_i\}$ and $\l_j> \l_k$ $\forall j\neq k.$ (It is not clear if the assumption that there is a unique $k$ such that $\lambda_k=\lambda_{\min}$ is always satisfied for rational VOAs.)    Then $\qdim_V M^i=\frac{S_{i,k}}{S_{0,k}}.$ The same consideration also applies to several other results including Propositions \ref{global}, \ref{simple current}, and Theorem \ref{possible value}.  below with suitable modifications.
\end{rem}

By using Verlinde Formula, we get the following result about global dimensions.
\begin{prop}\label{global} Let V be as in Lemma \ref{lem qdim}, the global dimension of $V$ is given by $\glob(V)=\frac{1}{S_{0,0}^2}.$
\end{prop}
\begin{proof}
By Lemma \ref{lem qdim},
$$\glob(V)=\sum_{i=0}^d\big{(}\frac{S_{i,0}}{S_{0,0}}\big{)}^2
=\sum_{i=0}^d \frac{S_{i,0}^2}{S_{0,0}^2}=\frac{\sum_{i=0}^d S_{i,0}^2}{S_{0,0}^2}.$$
By Theorem \ref{Verlinde} (1) and (2), $\sum_{i=0}^d S_{i,0}^2=1$ and the result follows.
\end{proof}

Next, we turn our  discussion to the quantum dimension of tensor product of two modules.
\begin{de} Let $V$ be a vertex operator algebra, and $M^1$, $M^2$ be two  $V$-modules.
A module $(W, I)$, where
$I\in I_V\left (
\begin{array}{c}
\ \ W \ \\
M^1 \ \  M^2
\end{array}
 \right ),$ is called a tensor product of $M^1$ and $M^2$ if for any $V$-module $M$ and $\mathcal{Y}\in
I_V\left (
\begin{array}{c}
\ \ M \ \\
M^1 \ \  M^2
\end{array}
 \right ) $, there is a unique $V$-module homomorphism $f: W\rightarrow M$, such that $\mathcal{Y}=f\circ I$. As usual, we denote $(W,I)$ by $M^1 \boxtimes M^2$.
\end{de}

\begin{rem}\label{rem tensor}
It is well known that for a rational \voa $V,$ let $M^0,\cdots, M^d$ be all irreducible $V$-modules,  then $M^i \boxtimes M^j$ exists and $$M^i \boxtimes M^j= \sum_{k=0}^d N_{i,j}^k M^k,$$
where $N_{i,j}^k$ are the fusion rules.

Moreover, if $V$ is also $C_2$-cofinite, then the tensor product is commutative and associative and $N_{i,j}^k=N_{i,j'}^{k'}$ (see \cite{HL1}-\cite{HL4}, \cite{Li3}).

\end{rem}

\begin{rem}\label{V*M}
It is easy to show that $V\boxtimes M=M$ for any $V$-module $M$.
\end{rem}

\begin{prop} \label{prop tensor}
Let V and $\{M^0,M^1,\ldots, M^d\}$ be as in Lemma \ref{lem qdim}, and also assume $V\cong V',$ then $$\qdim_V (M^i\boxtimes M^j)= \qdim_V M^i\cdot \qdim_V M^j.$$
\end{prop}
\begin{proof}
By Lemma \ref{lem qdim} and  Remark \ref{rem tensor}, it suffices to show $$\sum_{k=0}^d N_{i,j}^k \frac{S_{k,0}}{S_{0,0}}= \frac{S_{i,0}}{S_{0,0}}\frac{S_{j,0}}{S_{0,0}}.$$
Using Theorem \ref{Verlinde} gives
\begin{equation*}
\begin{split}
\sum_{k=0}^d N_{i,j}^k \frac{S_{k,0}}{S_{0,0}}=& \sum_{k,s=0}^d \frac{S_{i,s}S_{j,s}S^{-1}_{s,k}}{S_{0,s}}\frac{S_{k,0}}{S_{0,0}} \\
=& \frac{1}{S_{0,0}}\sum_{s=0}^d \bigg (\frac{S_{i,s}S_{j,s}}{S_{0,s}}\sum_{k=0}^d \big (S^{-1}_{s,k}S_{k,0}\big )\bigg ) \\
=&\frac{1}{S_{0,0}}\sum_{s=0}^d \big (\frac{S_{i,s}S_{j,s}}{S_{0,s}} \delta_{s,0}\big ) \ \ \ \ \ \ \ \ \ \ \  \\
=&\frac{1}{S_{0,0}}\frac{S_{i,0}S_{j,0}}{S_{0,0}}\ \ \ \ \ \ \ \ \ \ \ \ \ \ \ \ \ \ \ \ \ \\
=&\frac{S_{i,0}}{S_{0,0}}\frac{S_{j,0}}{S_{0,0}}.
\end{split}
\end{equation*}
The proof is complete.
\end{proof}

\begin{rem}
Under the assumptions of Remark \ref{nonunitary},  one can easily prove that quantum dimensions are also multiplicative under tensor product using the Verlinde formula.
\end{rem}

We now turn our attention to the quantum dimensions of simple currents.
\begin{de} Let V be a simple vertex operator algebra, a simple $V$-module $M$ is called a simple current if for any irreducible $V$-module $W$, $M\boxtimes W$ exists and is also a simple $V$-module.
\end{de}
It is clear from the definition that the simple current is an analogue of $1$-dimensional module for groups. Our goal is to establish that $M$ is a simple current if and only if $\qdim_VM=1.$

\begin{rem} By Remark \ref{V*M}, if $V$ is a simple vertex operator algebra, then $V$ itself is a simple current.
\end{rem}

\begin{lem}\label{tensor nonzero}
Let $V$ be a rational and $C_2$-cofinite simple vertex operator algebra, and $M$, $N$ be two admissible $V$-modules. Then $M\boxtimes N \neq 0.$
\end{lem}
\begin{proof}
As $V$ is rational and $C_2$-cofinite, by Remark \ref{rem tensor}, we know the tensor product of any two $V$-modules exists.
Thus we can consider $$\big (M\boxtimes N\big )\boxtimes N'=M\boxtimes \big (N\boxtimes N'\big )\supseteq M \boxtimes V= M\neq 0,$$
which implies $M\boxtimes N\neq 0.$
\end{proof}

\begin{lem}\label{>1}
Let $V$ be a rational and $C_2$-cofinite simple \voa of $CFT$ type with $V\cong V'$, and let $M^0,\ M^1,\ \cdots ,\ M^d$ be as before with the corresponding conformal weights $\l_i>0,$ $0<i\leq d$. For  any irreducible $V$-module $M$, $$\qdim_V M\geq 1.$$
\end{lem}

\begin{proof}
Let $M^i$ be an irreducible $V$-module with minimal quantum dimension. By Lemma \ref{lem qdim}, obviously $\qdim_V M^i>0$ and $\qdim_V M>0$.  Then using Propsition \ref{prop tensor}, we get
\begin{equation*}
\qdim_V M^i \cdot \qdim_V M= \qdim_V \big ( M^i\boxtimes M\big )> 0.
\end{equation*}
$M^i$ was chosen to be of minimal quantum dimension, therefore
\begin{equation*}
\qdim_V M^i \cdot \qdim_V M\geq \qdim_V M^i> 0.
\end{equation*}
Thus we get $\qdim_V M\geq 1.$
\end{proof}

\begin{lem}\label{simple current one}
Let $V$ be a rational and $C_2$-cofinite simple \voa of $CFT$ type with $V\cong V'$, and let $M^0,\ M^1,\ \cdots ,\ M^d$ be as before with the corresponding conformal weights $\l_i>0,$ $0<i<d$. If  $M$ is a simple current of $V$, then $\qdim_V M=1$.
\end{lem}
 \begin{proof}
 Let $M$ be a simple current of $V$.  Immediately we have $M\boxtimes M'=V$. Then by Remark \ref{M'} and Proposition \ref{prop tensor} we have
\begin{equation*}
1=\qdim_V V=\qdim_V (M\boxtimes M')=\qdim_V M\cdot \qdim_V M'=(\qdim_V M)^2.
\end{equation*}
 This implies
 $\qdim_V M=1.$
 \end{proof}

\begin{lem}\label{lem every occur}
Let V be a rational and $C_2$-cofinite simple vertex operator algebra, and $M^0\cong V, M^1, \cdots, M^d$ be as before.  Fix an irreducible $V$-module $M^s,$ then for any $0\leq i\leq d,$ we have
\begin{equation*}
M^i\subseteq M^s\boxtimes M^j
\end{equation*}
for some $0\leq j \leq d.$
\end{lem}
 \begin{proof}
 Again, the rationality and $C_2$-cofiniteness guarantee that tensor products are well defined and associative. So
\begin{equation*}
M^i\subseteq (M^s\boxtimes M^{s'})\boxtimes M^i=M^s\boxtimes (M^{s'}\boxtimes M^i),
\end{equation*}
where $M^{s'}$ is the dual of $M^s.$
Note that  $M^{s'}\boxtimes M^i=\oplus_{j=0}^dN_{s',i}^jM^j.$ Thus
\begin{equation*}
M^i\subseteq \bigoplus_{j=0}^dN_{s',i}^j M^s\boxtimes M^j.
\end{equation*}
Since $M^i$ is simple,  $M^i\subseteq M^s\boxtimes M^j$ for some $j.$
\end{proof}

 \begin{prop}\label{simple current}
Let $V$ be a rational and $C_2$-cofinite simple \voa of $CFT$ type with $V\cong V'$, and let $M^0,\ M^1,\ \cdots ,\ M^d$ be as before with the corresponding conformal weights $\l_i>0,$ $0<i\leq d$, then a $V$-module $M$ is a simple current if and only if $\qdim_VM =1$.
\end{prop}
 \begin{proof}
 By Lemma \ref{simple current one} any simple current has quantum dimension 1. Now assume $M$ is a $V$-module such that $\qdim_VM =1$. Obviously, Lemma \ref{>1} shows that $M$ is simple.
 Lemma \ref{lem every occur} claims that
  \begin{equation*}
  \bigoplus_{i=0}^d M^i \subseteq \bigoplus_{i=0}^d M \boxtimes M^i.
  \end{equation*}
This implies
\begin{equation*}
\qdim_V(\bigoplus_{i=0}^d M^i)\leq \qdim_V\bigoplus_{i=0}^d M \boxtimes M^i.
\end{equation*}

Since $\qdim_V M=1,$ computing the quantum dimensions on both sides of the above equation gives
  \begin{equation*}
  \qdim_V (\bigoplus_{i=0}^d M^i)= \sum_{i=0}^d 1\cdot \qdim_V M^i=\sum_{i=0}^d \qdim_V M\boxtimes M^i=\qdim_V (\bigoplus_{i=0}^d M \boxtimes M^i).
  \end{equation*}
  We obtain $\oplus_{i=0}^d M^i = \oplus_{i=0}^d M \boxtimes M^i.$

By Lemma \ref{tensor nonzero}, we can conclude that for any $M^i$, $i=0, \ \ldots,\ d$, there exists some $j=0, \ \ldots,\ d$, such that $M\boxtimes M^i= M^j$, i.e. $M$ is a simple current.
\end{proof}

\subsection{Examples}
 Now, by using the properties above we can compute more examples about quantum dimensions.

\begin{ex}\label{ex affine} {\rm Let $\frak{g}$ be a finite dimensional simple Lie algebra with Cartan subalgebra $\mathfrak{h}$, and $\hat{\frak{g}}$ be the corresponding affine Lie algebra.
Fix a positive integer $k.$ For any $\l \in \mathfrak{h}^*,$ we denote the corresponding irreducible highest weight module for $\hat{\frak{g}}$ by $L_{\frak{g}}(k,\l ).$ It is proved in \cite{DL}, \cite{FZ}, \cite{Li2} that  $L_{\frak{g}}(k,0)$ is a rational simple \voa and all irreducible $L_{\frak{g}}(k,0)$-modules are classified as $\{L_{\frak{g}}(k,\l )|\langle\l ,\theta\rangle\leq k,\l\in \mathfrak{h}^* \hbox{ is a dominant integral weight}\},$ where $\theta$ is the longest root of $\frak{g},$ and $(\theta,\theta)=2.$

The quantum dimensions of the irreducible $L_{\frak{g}}(k,0)$-modules can be computed using Lemma \ref{lem qdim}
and the modular transformation property of affine characters \cite{K}. The formula is given in \cite{C}:
\begin{equation*}
\qdim_{L_{\frak{g}}(k,0)}L_{\frak{g}}(k,\lambda)=\prod_{\alpha>0}\frac{\langle \lambda+\rho,\alpha\rangle_q}{\langle \rho,\alpha\rangle_q},
\end{equation*}
where $\rho$ is the Weyl vector,
$\alpha$ belongs to the set of positive roots and $n_q=\frac{q^n-q^{-n}}{q-q^{-1}}$, where $q=e^{i\pi/(k+\check{h})}$ and $\check{h}$ is the dual coxeter number of $\frak{g}$. In the operator algebra framework, the statistic dimension of $L_{\frak{g}}(k,\lambda)$ is given by the same formula when $\frak{g}$ is of type $A$ in \cite{Wa}.}
\end{ex}

\begin{ex} {\rm Let $V_L$ be the lattice \voa associated to an positive definite even lattice $L$. It is proved in \cite{D}, \cite{DLM4} that $V_L$ is rational and $C_2$-cofinite, and $\{V_{L+\mu_i}|i\in L^\circ/L\}$ gives a complete list of all irreducible $V_L$-modules, where $L^\circ$ is the dual lattice of $L$. By \cite{DL}, every irreducible $V_L$-module is a simple current. Thus
\begin{equation*}
\qdim_{V_L}V_{L+\mu_i}=1
\end{equation*}
by Proposition \ref{simple current}.

This implies a well known property on the theta functions of lattices. Note that $\ch_q V_{L+\mu_i}=\frac{\theta_{L+\mu_i}(q)}{\eta(q)^d}$
where $\theta_{L+\mu_i}(q)=\sum_{\alpha\in L+\mu_i}q^{(\alpha,\alpha)/2}$ and $\eta(q)=q^{1/24}\prod_{n\geq 1}(1-q^n)$ and $d$ is the rank of $L.$ Then
$$\lim_{q\to 1}\frac{\theta_{L+\mu_i}(q)}{\theta_L(q)}=\lim_{q\to 1}\frac{\ch_qV_{L+\mu_i}}{\ch_qV_L}=1. $$}
\end{ex}

\begin{ex} {\rm
We now consider the \voa $L(\frac{1}{2},0)$ associated to the Virasoro algebra with central charge $\frac{1}{2}$. This is a rational and $C_2$-cofinite \voa with only three irreducible modules $L(\frac{1}{2},0)$, $L(\frac{1}{2},\frac{1}{2})$, and $L(\frac{1}{2},\frac{1}{16})$ (see \cite{DLM4}, \cite{DMZ}, \cite{W}). We use two different methods to compute the quantum dimensions of the modules.

(1) We first compute the  quantum dimensions by using the $S$-matrix (see Lemma \ref{lem qdim}).  For short we use  $Z_0(\tau),$ $Z_{\frac{1}{2}}(\tau),$ $Z_{\frac{1}{16}}(\tau)$ to denote  $Z_{L(\frac{1}{2},0)}(\t)$, $Z_{L(\frac{1}{2}, \frac{1}{2})}(\t)$, and $Z_{L(\frac{1}{2},\frac{1}{16})}(\t)$ respectively. The modular transformation rules below is given in \cite{Ka}:
\begin{equation*}
\begin{split}
Z_0(-\frac{1}{\tau})&=\frac{1}{2}Z_{0}(\tau)+\frac{1}{2}Z_{\frac{1}{2}}(\tau)+\frac{\sqrt{2}}{2}Z_{\frac{1}{16}}(\tau),\\
Z_{\frac{1}{2}}(-\frac{1}{\tau})&=\frac{1}{2}Z_{0}(\tau)+\frac{1}{2}Z_{\frac{1}{2}}(\tau)-\frac{\sqrt{2}}{2}Z_{\frac{1}{16}}(\tau),\\
Z_{\frac{1}{16}}(-\frac{1}{\tau})&=\frac{\sqrt{2}}{2}Z_{0}(\tau)-\frac{\sqrt{2}}{2}Z_{\frac{1}{2}}(\tau).
\end{split}
\end{equation*}
By Lemma \ref{lem qdim}, we get
\begin{equation}
\begin{split}
\qdim_{L(\frac{1}{2},~0)}{L(\frac{1}{2},~0)}&=1,\\
\qdim_{{L(\frac{1}{2},~0)}}L(\frac{1}{2},~\frac{1}{2})&=\frac{1/2}{1/2}=1,\\
\qdim_{{L(\frac{1}{2},~0)}}L(\frac{1}{2},\frac{1}{16})&=\frac{\sqrt{2}/2}{1/2}=\sqrt{2}.
\end{split}
\end{equation}

(2) We can also compute the quantum dimensions by using Proposition \ref{prop tensor}. The fusion rules for these modules are well known (see \cite{DMZ}, \cite{W}):
 \begin{equation}\label{fusion Vir}
\begin{split}
&(1) ~~~L(\frac{1}{2},0){\rm ~is ~the~identity};\\
& (2) ~~~L(\frac{1}{2},\frac{1}{2})\times L(\frac{1}{2},\frac{1}{2})=L(\frac{1}{2},0);\\
& (3) ~~~L(\frac{1}{2},\frac{1}{2})\times L(\frac{1}{2},\frac{1}{16})=L(\frac{1}{2},\frac{1}{16});\\
&(4)  ~~~ L(\frac{1}{2},\frac{1}{16})\times L(\frac{1}{2},\frac{1}{16})=L(\frac{1}{2},0)+L(\frac{1}{2},\frac{1}{2}).
\end{split}
\end{equation}
Obviously, $L(\frac{1}{2},0)$ and $L(\frac{1}{2},\frac{1}{2})$ are simple currents, so we have
\begin{equation}\label{0,1/2}
\qdim_{L(\frac{1}{2},0)}L(\frac{1}{2},0)=\qdim_{L(\frac{1}{2},0)}L(\frac{1}{2},\frac{1}{2})=1.
\end{equation}
Using (\ref{fusion Vir}) (\ref{0,1/2}) and Proposition \ref{prop tensor} we have
\begin{equation*}
\big{[}\qdim_{L(\frac{1}{2},~0)}L(\frac{1}{2},\frac{1}{16})\big{]}^2=1+1=2.
\end{equation*}
That implies
\begin{equation*}
\qdim_{L(\frac{1}{2},~0)}L(\frac{1}{2},\frac{1}{16})=\sqrt{2}.
\end{equation*}}
\end{ex}

\begin{ex}
{\rm Let $L(c,h)$ be the irreducible Virasoro algebra module as in Example \ref{dim L(1,h)}. It is proved in \cite{W} that the \voa $L(c,0)$ is rational if and only if $c=c_{p,q}=1-\frac{6(p-q)^2}{pq},$ where $p,q\in \{2,3,4,\cdots\},$ and $p, q$ are relatively prime. And all irreducible $L(c_{p,q},0)$-modules are classified as $L(c_{p,q},h_{m,n})$ for $0<m<p,$ $0<n<q,$ where $h_{m,n}=\frac{(np-mq)^2-(p-q)^2}{4pq}.$   The fusion rules are also determined in \cite{W}.

Let $\chi_{m,n}^{p,q}(\tau)$ denotes the character of $L(c_{p,q},h_{m,n}).$ The $S$-modular transformation of characters has the following form \cite{CIZ1}, \cite{CIZ2}, \cite{IZ}:
\begin{equation*}
\chi_{m,n}^{p,q}(-1/\tau)=\sum_{m',n'}S_{m,n}^{m',n'}\chi_{m',n'}^{p,q}(\tau),
\end{equation*}
where
\begin{equation}\label{S Vir}
S_{m,n}^{m',n'}=\sqrt{\frac{8}{pq}}(-1)^{m'n+n'm+1}\sin(\frac{\pi mm'q}{p})\sin(\frac{\pi nn'p}{q}).
\end{equation}

 The case $p=q-1$ gives the unitary discrete series, which has been extensively studied in physics and conformal nets theory. Using Lemma \ref{lem qdim} and formula \ref{S Vir} one can easily compute the quantum dimensions of irreducible modules for the unitary discrete series. For the nonunitary case, $\l_{\min}=\frac{1-(p-q)^2}{4pq}\neq 0$   (see \cite{DM3}). In this case, one needs to find $m,n$ such that $|mq-np|=1$.  We only discuss two special cases $(p,q)=(2,5)$ and $(3,5)$ here.  The quantum dimensions of their irreducible modules are computed using the $S$ matrix and Remark \ref{nonunitary}.

1) The case $(p,q)=(2,5):$ $c_{p,q}=-\frac{22}{5}$ and $L(c_{2,5},0)$ has 2 irreducible modules $L(c_{2,5}, h_{1,n}),$ $n=1,2.$ Here $h_{1,1}=0,$ and $h_{1,2}=-\frac{1}{5}=\l_{\min}.$ A straightforward calculation gives
$$S_{1,1}^{1,2}=\sqrt{\frac{4}{5}}\sin (\frac{4\pi}{5}),~ S_{1,2}^{1,2}=-\sqrt{\frac{4}{5}}\sin (\frac{8\pi}{5}).$$
Thus $$\qdim_{L(c_{2,5},0)}L(c_{2,5},h_{1,1})=1,$$ and $$\qdim_{L(c_{2,5},0)}L(c_{2,5},h_{1,2})= \frac{ S_{1,2}^{1,2}}{S_{1,1}^{1,2}}=-\frac{\sqrt{\frac{4}{5}}\sin (\frac{8\pi}{5})}{\sqrt{\frac{4}{5}}\sin (\frac{4\pi}{5})}=2\cos\frac{\pi}{5}.$$

2) The case $(p,q)=(3,5):$ $c_{p,q}=-\frac{3}{5}$ and $L(c_{3,5},0)$ has 4 irreducible modules $L(c_{3,5},h_{1,n}),$ $n=1,2,3,4.$  Here $h_{1,1}=0,$ $h_{1,2}=-\frac{1}{20}=\l_{\min},$ $h_{1,3}=\frac{1}{5}$ and $h_{1,4}=\frac{3}{4}.$
By using (\ref{S Vir}), it is easy to get
\begin{equation*}
\begin{split}
&S_{1,1}^{1,2}=\sqrt{\frac{8}{15}}\sin(\frac{5\pi}{3})\sin(\frac{6\pi}{5});\\
&S_{1,2}^{1,2}=-\sqrt{\frac{8}{15}}\sin(\frac{5\pi}{3})\sin(\frac{12\pi}{5});\\
&S_{1,3}^{1,2}=\sqrt{\frac{8}{15}}\sin(\frac{5\pi}{3})\sin(\frac{18\pi}{5});\\
&S_{1,4}^{1,2}=-\sqrt{\frac{8}{15}}\sin(\frac{5\pi}{3})\sin(\frac{24\pi}{5}).
\end{split}
\end{equation*}
This implies
\begin{equation*}
\begin{split}
&\qdim_{L(c_{3,5},0)}L(c_{3,5},h_{1,1})=\qdim_{L(c_{3,5},0)}L(c_{3,5},h_{1,4})=1;\\
&\qdim_{L(c_{3,5},0)}L(c_{3,5},h_{1,2})=\frac{S_{1,2}^{1,2}}{S_{1,1}^{1,2}}=2\cos(\frac{\pi}{5});\\
&\qdim_{L(c_{3,5},0)}L(c_{3,5},h_{1,3})=\frac{S_{1,3}^{1,2}}{S_{1,1}^{1,2}}=2\cos(\frac{\pi}{5}).
\end{split}
\end{equation*}
}
\end{ex}

In the following example, the \voa $V$ is neither rational nor $C_2$-cofinite. But the quantum dimensions of its modules are still multiplicative under tensor product.

\begin{ex}\label{L(1,0)}{\rm  Recall $L(1,0)$ and $L(1,h)$ from Example \ref{dim L(1,h)}.
The fusion rules for irreducible $L(1,0)$-modules are given in \cite{M} and \cite{DJ}. Assume  $m,n\in \mathbb{Z}_+$ and $m\geq n$, then
\begin{equation*}
L(1,m^2)\boxtimes L(1,n^2)=\bigoplus_{k=m-n}^{m+n}L(1,k^2).
\end{equation*}
By Example \ref{dim L(1,h)}, we have
\begin{equation*}
\begin{split}
 &\qdim_{L(1,0)}(L(1,m^2) \boxtimes L(1,n^2))\\
 =&\qdim_{L(1,0)}({\bigoplus_{k=m-n}^{m+n}L(1,k^2)})\\
=&\sum_{k=m-n}^{m+n}(2k+1)\\
=& (2m+1)(2n+1)\\
=&\qdim_{L(1,0)}L(1,m^2)\cdot \qdim_{L(1,0)}L(1,n^2).
\end{split}
\end{equation*}
From the fusion rules, the only simple current among $L(1,m^2)$  is $L(1,0),$ which has quantum dimension $1.$}
\end{ex}

By Example \ref{L(1,0)}, it seems that even for \voas \ that are not rational, Propositions \ref{prop tensor} and \ref{simple current} are still  true. But the proofs of  these propositions require that $V$ is rational and $C_2$-cofinite as the modularity of trace functions are used. This gives us a good reason to believe that there might be an alternating definition for quantum dimensions which does not involve the trace functions. We believe that the other two limits given in Remark \ref{equidef1} are in the right direction.

\section{Possible Values of Quantum Dimensions}

In this section we will give a result on possible values of quantum dimensions using the graph theory and
Perron-Frobenius Theorem. It turns out that the values of quantum dimensions are closely related to the possible values of the index of subfactors \cite{J}.

\begin{de}
Let $A$ be an $n\times n$ matrix over $\mathbb{R}$, and $\l_1, ..., \l_n$ be all the eigenvalues of A. Then its spectral radius $\rho(A)$ is defined as:
\begin{equation*}
\rho(A)=\max_i(|\l_i|).
\end{equation*}
\end{de}

Next result is part of Perron-Frobenius Theorem (cf. \cite{BH}).

\begin{thm}
Let $A = (a_{ij})$ be an $n\times n$ positive matrix: $a_{ij} > 0,$ for $1 \leq i, j \leq n.$ Then the following statements hold:

(1) There is a positive real number r, 
 such that r is an eigenvalue of A and any other eigenvalue $\l$ (possibly complex) is strictly smaller than r in absolute value, $ |\l| < r,$ i.e. $\rho(A)=r$;

(2) There exists an eigenvector $v = (v_1, \cdots,v_n)$ of A with eigenvalue r such that all components of $v$ are positive;

(3) There are no other positive (moreover non-negative) eigenvectors except positive multiples of v, i.e. all other eigenvectors must have at least one negative or non-real component.
\end{thm}

\begin{rem}\label{rem PF}
In the case that $A$ is non-negative, we can use positive matrices to approach $A.$ So if there is a positive eigenvector of $A$ with positive eigenvalue $r,$ then $r=\rho(A).$
\end{rem}

\begin{rem}\label{Eigenvector}
Let $V$ be a simple \voa as in Lemma \ref{lem qdim} and $M^i$ for $i=0,\cdots, d$ be the irreducible $V$-modules as before. Theorem \ref{Verlinde} and Lemma \ref{lem qdim} assert that $\qdim_V M^i$ is a positive eigenvalue of $N(i)$ with eigenvector $v=(\frac{S_{0,0}}{S_{0,0}},\frac{S_{0,1}}{S_{0,0}},\cdots,\frac{S_{0,d}}{S_{0,0}})^T$ whose entries are all positive (since each component is a quantum dimension). By Remark \ref{rem PF}, one knows that $\qdim_V M^i$ is exactly the spectral radius of $N(i),$ i.e. $\qdim_V(M^i)=\rho(N(i)).$
\end{rem}

The following lemmas are devoted to proving the main result of this section.
\begin{lem} Let $V$ be a \voa as in Lemma \ref{lem qdim}, and $N(i)=(N_{i,j}^k)_{j,k}$ be the fusion matrix. Then $N(i)^T=N(i').$
\end{lem}
\begin{proof} Using Theorem \ref{Verlinde} gives

\begin{equation*}
\begin{split}
N_{i,j}^k&=\sum_{s=0}^d\frac{S_{j,s}S_{i,s}S^{-1}_{s,k}}{S_{0,s}}\\
&=\sum_{s'=0}^d\frac{S_{j',s'}S_{i',s'}S^{-1}_{s',k'}}{S_{0,s'}}\\
&=\sum_{s=0}^d\frac{S_{j',s}S_{i',s}S^{-1}_{s,k'}}{S_{0,s}}\\
&=N_{i',j'}^{k'}=N_{i',k}^j.
\end{split}
\end{equation*}
The proof is complete.
\end{proof}

Since the fusion algebra is commutative,  the following corollary is now obvious.
\begin{coro}\label{normal} The fusion matrix $N(i)$ is normal, i.e.
$N(i)^T N(i)=N(i)N(i)^T.$
\end{coro}

\begin{lem}\label{norm}
The matrix $
\left(
  \begin{array}{cc}
    0 \ & N(i)\\
    N(i)^T & 0\\
  \end{array}
\right)
$ is a symmetric matrix whose spectral radius equals to $\qdim_V M^i.$
\end{lem}
\begin{proof}
 It is clear that $\bar{N}(i)=
\left(
  \begin{array}{cc}
    0 \ & N(i)\\
    N(i)^T & 0\\
  \end{array}
\right)$ is a symmetric matrix. And the matrices $N(i)$ and $N(i)^T=N(i')$ have the same spectral radius $\qdim_V M^i$ with the same eigenvector $v$ (Remark \ref{Eigenvector}). Thus the vector $\left(
  \begin{array}{c}
   v\\
v
  \end{array}
\right)$ is an eigenvector of $\bar{N}(i)$ with eigenvalue $\qdim_V M^i.$ Again by Remark \ref{rem PF},  $\rho(\bar{N}(i))=\qdim_VM^i.$
\end{proof}

 In graph theory, the adjacency matrix of a finite graph $G$ on $n$ vertices is the $n\times n$ matrix where the non-diagonal entry $a_{ij}$ is the number of edges from vertex $i$ to vertex $j$, and the diagonal entry $a_{ii}$  is the number of loops from vertex $i$ to itself. Thus any symmetric matrix with all entries non-negative integers is the adjacency matrix of a certain finite graph.

\begin{de}
Let $G$ be a graph and $M$ be its adjacency matrix. The spectral norm of $G$ is defined as the spectral radius of its adjacency matrix $M,$ which is denoted by $\parallel G\parallel.$
\end{de}

\begin{rem}
If $G_i,$ $i=1,\cdots, k,$ are all the connected subgraphs of $G,$ then
\begin{equation*}
\max_i\parallel G_i\parallel= \parallel G\parallel.
\end{equation*}
 \end{rem}

\begin{thm}\label{possible value}
Let $V$ be a simple \voa  and $\{M^0, \cdots,\ M^d\}$ the  inequivalent irreducible $V$-modules as in Lemma \ref{lem qdim}. Then for any $0\leq i\leq d,$ $$\qdim_VM^i\in \{2\cos(\pi/n)| n\geq 3\}\cup \{a|2\leq a< \infty, \hbox{a is algebraic}\}.$$ 
\end{thm}
\begin{proof}
Since $\qdim_VM^i$ is the eigenvalue of the fusion matrix $N(i)$ whose entries are all nonnegative integers, it is an algebraic number.

 By Lemma \ref{norm}, we know that the matrix
 $$       \bar{N}(i)=
\left(
  \begin{array}{cc}
    0 \ & N(i)\\
    N(i)^T & 0\\
  \end{array}
\right)$$
is symmetric with all entries nonnegative integers. Thus $\bar{N}(i))$ is the adjacency matrix of a certain graph, say, $G_{i}$, whose norm is actually the quantum dimension of $M^i.$

 If $\parallel G_i\parallel\geq 2,$ we are done. If $\parallel G_i\parallel<2,$ it follows from \cite{S} and \cite{LS} that each connected subgraph of $G_i$ is of $ADE$ type.  And the spectral norms of $ADE$ type graphs  are of the form $2\cos(\pi/n)$ with  $n\geq 3.$
\end{proof}

\begin{rem} The possible values for the index of subfactors given in \cite{J} are
$$\{4\cos^2(\pi/n)|n\geq 3\}\cup [4,\infty).$$
So the possible values for quantum dimensions given in the previous theorem are exactly the square root of index of subfactors but restricted to algebraic numbers.
\end{rem}

\begin{rem}
The possible values that are less than 2 can be realized. Take $\frak{g}=sl_2(\mathbb{C}),$ and let $\l $ be the fundamental weight. Consider different levels $k\in\Z_+.$
 By Example \ref{ex affine}, one easily sees that
$$\qdim_{L_{\frak{g}}(k,0)}L_{\frak{g}}(k,\l)=2\cos (\frac{\pi}{k+2}), ~ k\geq 1.$$
But we do not know which number greater than or equal to 2 can be realized as a quantum dimension. According to \cite{ENO}, the quantum dimensions for rational and $C_2$-cofinite \voas \ are cyclotomic integers.
\end{rem}

\section{Quantum Galois Theory}

In this section, we study the Galois theory for vertex operator algebras. In the classical Galois theory, we need the degree $[E:F]$ of a field  $E$ over a subfield $F,$ which is defined as $\dim_FE.$  To have a Galois theory for \voas, we also need to define the degree of a vertex operator algebra $V$ over a vertex operator subalgebra $U.$  We define the degree $[V:U]$ to be the quantum dimension $\qdim_U V.$ In fact, exhibiting
a quantum Galois theory is one of the main motivations for us to study the quantum dimensions.

We need the following theorems for further discussions.
\begin{thm}\label{thm compact}
Let $V$ be a simple vertex operator algebra, $G$ be a compact subgroup of $\Aut(V)$, and $G$ acts continuously on $V$. Then as a $(G, V^G)$-module, $V=\oplus_{\chi \in \Irr(G)} (W_\chi \otimes V_ \chi)$,  where

(1) $V_\chi\neq 0$, $\forall \chi \in \Irr(G)$,

(2) $V_\chi$ is an irreducible $V^G$-module, $\forall \chi \in \Irr(G)$,

(3) $V_ \chi\cong V_\lambda$ as $V^G$-modules if and only if $\chi=\lambda$.
\end{thm}

\begin{thm}\label{tg2}
Let $V$ be as in the previous theorem, and $G$ be a finite subgroup of $\Aut(V)$, then the map $H\to V^H$ gives a bijection between subgroups of $G$ and subVOAs of V containing $V^G$.
\end{thm}

Theorem \ref{thm compact} is established in \cite{DLM}. If $G$ is solvable, this result has been obtained previously in \cite{DM1}.  Theorem \ref{tg2} is given in \cite{DM1} and \cite{HMT}.

It is well known in the classical Galois theory that if $G$ is a finite automorphism group of field $E$ then $[E:E^G]=o(G).$ The next result is a vertex operator algebra analogue of this result.
\begin{thm}\label{thm chi(1)}
Let $V$ be a rational and $C_2$-cofinite simple vertex operator algebra.
Also assume that $V$ is $g$-rational and the conformal weight of any irreducible $g$-twisted $V$-module is positive except for $V$ itself for all $g\in G.$
Then  $[V:V^G]$ exists and equals to $o(G).$
\end{thm}
\begin{proof}
It is well known that $o(G)=\sum_ {\chi \in \Irr(G)} (\dim W_\chi)^2.$ By Theorem \ref{thm compact}, if $\qdim_{V^G}V_{\chi}$ exists for all $\chi\in \Irr(G)$ then $[V:V^G]$  exists and is equal to
$\sum_{\chi\in \Irr(G)}\dim W_\chi \cdot \qdim_{V^G}V_\chi.$ The theorem holds if $\qdim_{V^G}V_{\chi}={\chi(1)}=\dim W_{\chi}.$

By the orthogonality of characters of representations of a finite group, we notice that $\ch_q(V_\chi)=\frac{1}{o(G)}\sum_{g\in G}Z_V(1, \ g,\ q)\overline{\chi (g)}.$ By Theorem \ref{minvariance} we have
\begin{equation*}
\begin{split}
\qdim_{V^G}V_{\chi}=&\lim_{q\to 1^-}\frac{\sum_{g\in G}Z_V(1, \ g,\ q)\overline{\chi (g)}}
{\sum_{g\in G}Z_V(1, \ g,\ q)}\\
=&\lim_{\tau \to i\infty}\frac {\sum_{g\in G}{Z_V(1,g,\ -\frac {1}{\tau})\overline{\chi(g)}} }
{\sum_{g\in G}Z_V(1,\ g, \ -\frac{1}{\tau})}\\
=&\lim_{\tau \to i\infty}\frac{\sum_{g\in G,N_i\in \mathscr{M}(g)}S_{0,i}(1,g)Z_{N_i}(g,~1,\ \tau)\overline{\chi(g)}}
{\sum_{g\in G,N_i\in \mathscr{M}(g)}S_{0,i}(1,g)Z_{N_i}(g,~1,\ \tau)}\\
=&\lim_{q \to 0^+}\frac{\sum_{g\in G,N_i\in \mathscr{M}(g)}S_{0,i}(1,g)Z_{N_i}(g,~ 1,\ q)\overline{\chi(g)}}
{\sum_{g\in G,N_i\in \mathscr{M}(g)}S_{0,i}(1,g)Z_{N_i}(g,~1,\ q)}.
\end{split}
\end{equation*}
By the assumption, we know that
$$\lim_{q\to 0^+} q^{c/24}Z_{V}(1,~1,q)\neq 0$$
and
$$\lim_{q\to 0^+} q^{c/24}Z_{N_i}(g,~1,q)=0$$
 for any other $N_i.$  This implies that
\begin{equation*}
\qdim_{V^G}V_{\chi}={\chi(1)}=\dim W_{\chi},
\end{equation*}
and the theorem follows.
\end{proof}

\begin{rem}
In Theorem \ref{thm chi(1)}, under certain conditions, we proved for a finite group $G< \Aut(V)$,  $\qdim_{V^G}V_{\chi}=\chi(1).$ It seems that this result is still true for a compact group $G<\Aut(V)$ as assumed in \cite{DLM}. Here are some examples.
\end{rem}

\begin{ex}{\rm
Let $V_L$ be the lattice VOA associated to a positive definite even lattice $L$ of rank $d$ with a nondegenerate bilinear form $(~,~).$
Set $\mathfrak{h}=L\otimes_{\mathbb{Z}}\mathbb{R}$ and  and let $M(1)$ be the rank $d$ Heisenberg \voa associated to $\mathfrak{h}.$  Then $V_L\cong M(1)\bigotimes (\oplus_{\alpha \in L} \mathbb{ C}e^{\alpha})$ as linear spaces. Recall that $L^{\circ}$ is the dual lattice of $L.$ Then $L\otimes_{\mathbb{Z}}\mathbb{R}/L^\circ \cong T^n$ is a compact Lie group acting continuously on $V_L$
 in the following way:  for any $\beta \in \mathfrak{h},$
\begin{equation*}
\begin{split}
e^{2\pi i\beta(0)}:~~  V_L ~&\to ~V_L\\
 a\otimes e^{\alpha}&\mapsto e^{2\pi i(\beta, \alpha)} a\otimes e^{\alpha}.
\end{split}
\end{equation*}
Since $\beta(0)$ is a derivation of $V_L,$  $e^{2\pi i\beta(0)}$ is an automorphism of $V_L$ such that
$e^{2\pi i\beta(0)}=1$ if $\beta\in L^{\circ}.$ As a result the torus $T^n$ is a compact subgroup of  $\Aut(V_L).$  It is easy to see that $(V_L)^{T^n}=M(1).$  By Theorem \ref{thm compact} we have a decomposition of $V_L$:
\begin{equation*}
V_L=\bigoplus_{\alpha \in L} M(1,\alpha)\otimes \mathbb{C}e^\alpha,
\end{equation*}
where $M(1,\alpha)$ is an $M(1)$-module with weight $\alpha ,$ and $\mathbb{C}e^\alpha$ is an irreducible $T^n$-module. We have already known from Example \ref{hei} that $\qdim_{M(1)} M(1,\alpha)=1.$ That is,
$\qdim_{M(1)} M(1,\alpha)=\dim \mathbb{C}e^\alpha.$}
\end{ex}

\begin{ex}{\rm
Let $V=V_L$ be the lattice \voa associated with the root lattice $L=\Z\alpha$ of type $A_1$ where
  $(\alpha,\alpha)=2.$ It is well known that $SO(3)$ is a subgroup of $\Aut(V)$ and we have the following  decomposition \cite{DG}:
\begin{equation*}
V=\bigoplus _{m\geq 0} W_{2m}\otimes L(1, m^2),
\end{equation*}
where $L(1,m^2)$ is the highest weight module for the Virasoro VOA $L(1,0)$ with highest weight $m^2$, and $W_{2m}$ is the irreducible $2m+1$ dimensional highest weight module for $SO(3)$ with highest weight $m$. In particular, $V^{SO(3)}=L(1,0).$ By Example \ref{L(1,0)}, one gets
\begin{equation*}
\dim W_{2m}=\qdim_{L(1,0)} L(1,m^2)=2m+1.
\end{equation*}}
\end{ex}

Motivated by these two  examples, we make the following conjecture:

\begin{conj}  Let $V = (V, Y, 1, \omega)$ be a rational and $C_2$-cofinite simple vertex operator algebra, and $G$ be a subgroup of $\Aut (V).$ Assume that $G$ is a finite-dimensional compact Lie group which acts on $V$ continuously.  Then the decomposition
\begin{equation*}
V=\sum_{\chi \in \Irr(G)}W_{\chi}\otimes V_{\chi}
\end{equation*}
has the following property:
\begin{equation*}
\dim W_{\chi}=\qdim_{V^G}V_{\chi}.
\end{equation*}
\end{conj}

We now turn our attention to the Galois extensions in the theory of vertex operator algebra. We need some definitions first.

\begin{de} Let $U$ be a vertex operator subalgebra of $V$ with the same Virasoro element. $V$ is called a Galois extension of $U$ if there exists a finite group $G< \Aut(V)$ such that $U=V^G,$ and $\qdim_U V\leq o(G).$
\end{de}

For any VOA extension $V\supset U$ we can define the Galois group $\Gal(V/U)=\{g\in \Aut(V)\left|\ g|_{U}=Id\right \}$ as in the classical field theory. The following two theorems are our main results about Galois extensions.

\begin{thm}\label{galois 1}
Let $V$ be a simple vertex operator algebra, and $G< \Aut(V)$ a finite group. Then $$\Gal(V/V^G)=G.$$
\end{thm}
\begin{proof}
Obviously $G\subset \Gal(V/V^G)$. We now prove the containment $ \Gal(V/V^G)\subset G$ with the help of Hopf algebra.

Let $\C[G]$ be the group algebra associated to $G.$ Then $\C[G]$ is a cocommutative
Hopf algebra with comultiplication $\Delta,$  counit $\varepsilon$  and antipode $S:$
\begin{equation*}
\begin{split}
&\Delta: \mathbb{C}[G]\to \mathbb{C}[G] \otimes \mathbb{C}[G], ~\Delta(g)=g\otimes g, ~ \hbox{for } g \in G,\\
&\varepsilon: ~ \mathbb{C}[G]\to \mathbb{C},~~~~~~~~~~~~~~~~ \varepsilon(g)=1,~~~~~\hbox{ for } g\in G,\\
&S : \mathbb{C}[G]\ \to \mathbb{C}[G],~~~~~~~~~~~S(g)=g^{-1}, ~~\hbox{ for } g\in G.
\end{split}
\end{equation*}
 Recall that an element $g\in \C[G]$ with $\Delta(g)=g\otimes g$ and  $\varepsilon(g)=1$ is called a group-like element. It is well known that the set of group-like elements of $\C[G]$ is exactly $G$ itself.
 So it is enough to show that any $g\in \Gal(V/V^G)$ is a group-like element.

Since  $g|_{_{V^G}}=Id$,  $g:V\to V$ gives a $V^G$-module homomorphism. Thus
 $g W\chi \subseteq W\chi,$ for any $\chi\in \Irr (G),$ where $W_\chi$ is the same as in Theorem \ref{thm compact}. That implies
\begin{equation*}
g\in \bigoplus_{\chi \in \Irr(G)}\End (W_\chi).
\end{equation*}
Together with the fact that $\mathbb{C}[G]\cong \oplus_{\chi \in \Irr(G)}\End (W_\chi) $, $g$ can be viewed as an element in $\mathbb{C}[G]$.
We write $g=\sum_{h\in G} \lambda_{h}h,$ $\lambda_{h}\in \mathbb{C}.$
As $g|_{_{V^G}}=Id=\sum_{h\in G} \lambda_{h}h|_{V^G}=\sum_{h\in G}\lambda_{h}Id,$ we get
$\sum_{h\in G} \lambda_{h}=1,$ i.e. $\varepsilon(g)=1.$

Now in order to show $\Delta(g)=g\otimes g,$ it suffices to show that for any $\chi ,\ \gamma \in \Irr(G),$ $g(a\otimes b)=ga\otimes gb$, where $a\in W_\chi$ and $b\in W_{\gamma}$.
Let $W_\chi$, $W_\gamma$ be two $G$-submodules in $V.$ It is proved in \cite{DM2}  that there is a $G$-module isomorphism for sufficiently small $n$:
\begin{equation*}
\begin{split}
\psi_n: W_\chi \otimes W_\gamma &\to \langle\sum _{m=n}^{\infty}u_m v|u\in W_\chi,\ v\in W_\gamma\rangle,\\
u\otimes v & \mapsto \sum _{m=n}^{\infty}u_m v .
\end{split}
\end{equation*}
Since $g$ is an automorphism of $V$, $g(\sum _{m=n}^{\infty}u_m v)=\sum _{m=n}^{\infty}(gu)_m gv=\psi_n (g u\otimes gv)).$ We get $g (u\otimes v)=gu\otimes gv .$ Thus $g$ is a group-like element and the proof is complete.
\end{proof}

\begin{thm}\label{thm Galois}
 Let $V$ be a simple \voa as in Theorem \ref{thm chi(1)}, and $G$ a finite automorphism group of $V.$
 Then $H\mapsto V^H$ gives a one-to-one correspondence from the subgroups of $G$ and the vertex operator subalgebras of $V$ containing $V^G$ satisfying the following:

(1) For any subgroup $H$ of $G$, $[V:V^H]=o(H)$ and   $[V^H:V^G]=[G:H]$,

(2) $H\lhd G$ if and only if $V^H$ is a Galois extension of $V^G$. In this case $\Gal(V^H/V^G)\cong G/H.$
\end{thm}
\begin{proof} The one to one correspondence is given in Theorem \ref{tg2}. By Theorem \ref{thm chi(1)} with $G$ replaced by $H$ we easily see that
\begin{equation*}
[V:V^H]=\sum_{\chi \in \Irr(H)} \dim(W_{\chi})\cdot \chi(1)=\sum_{\chi \in \Irr(H)}\chi(1)^2=o(H).
\end{equation*}
Also, $[V^H:V^G]=[V:V^G]/[V:V^H]=[G:H]$ and this proves (1).

For (2), we first notice that $gV^H=V^{gHg^{-1}}$ for $g\in G.$ If $H\lhd G$,  $G/H$ acts naturally on $V^H.$ So  $G/H$ is a subgroup of $\Gal(V^H/V^G).$ It is clear that $V^G=(V^H)^{G/H}.$ Then by Theorem \ref{galois 1}, $\Gal(V^H/V^G)\cong G/H$. Together with the fact $[V^H:V^G]=[G:H],$ we conclude that $V^H$ is a Galois extension of $V^G$.

Now we assume $V^H$ is a Galois extension of $V^G.$ For short we set $G'=\Gal(V^H/V^G).$
By Part (1) and the definition of Galois extension we know that $(V^H)^{G'}=V^G$ and
\begin{equation}\label{<}
[G:H]=[V^H:V^G]=[V^H:(V^H)^{G'}]\leq o(G').
\end{equation}
By Theorem \ref{thm compact}, we have two  decompositions
\begin{equation*}\label{V}
V=\bigoplus_{\chi \in \Irr(G)}W_\chi \otimes V_\chi ,
\end{equation*}
\begin{equation}\label{V^H}
V^H=\bigoplus_{\chi \in \Irr(G)} R_{\chi}\otimes V_\chi,
\end{equation}
where $R_\chi \subseteq W_\chi$ is a subspace, and each $R_\chi$ is an irreducible $G'$-module. We also know every irreducible $G'$-module occurs in $V^H$ by Theorem \ref{thm compact}.

By Theorem \ref{thm chi(1)} and equation (\ref{V^H}), we have
\begin{equation*}
\begin{split}
[V^H:V^G]=&\sum_{\chi \in \Irr(G)}\dim R_\chi \cdot \qdim_{V^G} V_\chi\\
=&\sum_{\chi \in \Irr(G)}\dim R_\chi\cdot \dim W_\chi\\
\geq & \sum_{\chi \in \Irr(G)}(\dim R_\chi)^2\\
=&o(G').
\end{split}
\end{equation*}
Together with equation \eqref{<}, we have that for any $\chi \in \Irr(G)$ with $R_{\chi}\neq 0$,
$\dim R_{\chi}=\qdim_{V^G}V_\chi,$ i.e. either $R_\chi =0$ or $R_\chi = W_\chi,$ therefore for any $g\in G,$ $gV^H=V^{gHg^{-1}}\subset V^H.$
 By Theorem \ref{tg2},  $H\leqslant gHg^{-1},$ which implies $H=gHg^{-1}.$  And the proof is complete.
\end{proof}

\begin{rem} Let $E\supset F$ be two fields. In classical Galois theory,  the following two definitions for Galois extension are equivalent:

(1) $E$ is called a Galois extension of $F$ if $F=E^G$ for some $G,$ where $G$ is a finite subgroup of $\Aut(E),$

(2) $E$ is called a Galois extension of $F$  if $\dim_F E =\Gal(E:F).$

We believe that the same is true for vertex operator algebra. But we cannot prove it in this paper. However, if $V$ is a rational vertex operator algebra satisfying the assumptions given in Theorem \ref{thm chi(1)},
these two definitions are equivalent. Since $V^H$ in Theorem \ref{thm Galois} does not satisfy the assumptions of Theorem \ref{thm chi(1)} (these assumptions should hold according to conjectures in orbifold theory but have not been established), we need to use both (1) and (2) in the definition of Galois extension for vertex operator algebra to have a good Galois theory.
\end{rem}

\end{document}